\documentclass[12pt]{article}
\usepackage{amsmath,amssymb,amsthm, amsfonts}
\usepackage{etaremune, enumerate, float, verbatim}
\usepackage{algorithm}
\usepackage{algorithmic}
\usepackage{hyperref}
\usepackage{graphicx}
\usepackage{url}
\usepackage{dsfont} 
\usepackage{tikz, graphicx, subcaption, caption} 
\usepackage[algo2e,ruled,vlined]{algorithm2e}
\usepackage{mathrsfs}
\usepackage{soul}
\usepackage{float}
\usepackage{color}
\definecolor{red}{rgb}{1,0,0}

\definecolor{blue}{rgb}{0,0,.9}

\definecolor{green}{rgb}{0,.6,0}

\definecolor{purp}{rgb}{.5,0,.5}

\definecolor{grey}{rgb}{.3,.3,.3}

\numberwithin{figure}{section}
\numberwithin{table}{section}   
\numberwithin{equation}{section}   

\usepackage{tikz}
\usetikzlibrary{fit}
\usetikzlibrary{backgrounds}
\usetikzlibrary{shapes.geometric}
\tikzstyle{vertex}=[circle, draw=black, thick, inner sep=0pt, minimum size=6pt]
\tikzstyle{Bvertex}=[circle, black, fill, draw, inner sep=0pt, minimum size=6pt]
\tikzstyle{vtx}=[circle, white, fill, draw=black, thick, inner sep=0pt, minimum size=6pt]
\tikzstyle{gvertex}=[circle, green, fill, draw=black, inner sep=0pt, minimum size=6pt]

\newtheorem{thm}{Theorem}[section]
\newtheorem{cor}[thm]{Corollary}
\newtheorem{lem}[thm]{Lemma}
\newtheorem{prop}[thm]{Proposition}

\theoremstyle{definition}
\newtheorem{rem}[thm]{Remark}

\theoremstyle{definition}

\theoremstyle{definition}
\newtheorem{ex}[thm]{Example}

\newcommand{\Z}{\operatorname{Z}}

\newcommand{\Zp}{\operatorname{Z}_+}

\newcommand{\pd}{\gamma_P}

\newcommand{\ptz}{\operatorname{pt}_{\Z}}
\newcommand{\ptx}{\operatorname{pt}_Y}
\newcommand{\ptp}{\operatorname{pt}_+}

\newcommand{\ppt}{\operatorname{pt}_{\gamma_P}}

\newcommand{\thx}{\operatorname{th}_Y}

\newcommand{\thpd}{\operatorname{th}_{\gamma_P}}
\newcommand{\thxx}{\operatorname{th}_Y^\times}
\newcommand{\thxa}{\operatorname{th}_Y^\ast}
\newcommand{\thzx}{\operatorname{th}_{\Z}^\times}
\newcommand{\thza}{\operatorname{th}_{\Z}^\ast}
\newcommand{\thpx}{\operatorname{th}_+^\times}
\newcommand{\thpa}{\operatorname{th}_+^\ast}

\newcommand{\thpda}{\operatorname{th}_{\gamma_P}^\ast}
\newcommand{\thpdx}{\operatorname{th}_{\gamma_P}^\times}

\newcommand{\epn}{\operatorname{epn}}

\newcommand{\x}{\times}

\newcommand{\bit}{\begin{itemize}}
\newcommand{\eit}{\end{itemize}}
\newcommand{\ben}{\begin{enumerate}}
\newcommand{\een}{\end{enumerate}}
\newcommand{\beq}{\begin{equation}}
\newcommand{\eeq}{\end{equation}}
\newcommand{\bea}{\begin{eqnarray*}} 
\newcommand{\eea}{\end{eqnarray*}}
\newcommand{\bpf}{\begin{proof}}
\newcommand{\epf}{\end{proof}\ms}
\newcommand{\bmt}{\begin{bmatrix}}
\newcommand{\emt}{\end{bmatrix}}
\newcommand{\ms}{\medskip}

\newcommand{\cp}{\, \Box\,}
\newcommand{\lc}{\left\lceil}
\newcommand{\rc}{\right\rceil}
\newcommand{\lf}{\left\lfloor}
\newcommand{\rf}{\right\rfloor}

\newcommand{\noi}{\noindent}
\newcommand{\ol}{\overline}

\title{Sharp bounds for product and sum throttling numbers}
\author{Ryan Blair\thanks{Department of Mathematics and Statistics, California State University Long Beach, Long Beach, CA  90840, USA (ryan.blair@csulb.edu)}\and Gabriel Elvin\thanks{Department of Mathematics, California State University, San Bernardino, San Bernardino, CA 92407, USA (Gabriel.Elvin@csusb.edu)}\and Veronika Furst\thanks{Department of Mathematics, Fort Lewis College, Durango, CO  81301, USA (furst\_v@fortlewis.edu)}\and Leslie Hogben\thanks{American Institute of Mathematics, Pasadena, CA 91125, USA (hogben@aimath.org); Department of Mathematics, Iowa State University, Ames, IA 50011, USA; Department of Mathematics, Purdue University, West Lafayette, IN 47906, USA.}\and Nandita Sahajpal \thanks{Department of Data, Media and Design, Nevada State University, Henderson, NV 89002, USA (nandita.sahajpal@nevadastate.edu)}\and Tony W. H. Wong \thanks{Department of Mathematics, Kutztown University of Pennsylvania, Kutztown, PA 19530, USA (wong@kutztown.edu)}}

\begin{document}

\maketitle

\begin{abstract} 
Throttling in graphs optimizes a sum or product of resources used, such as the number of vertices in an initial set, and time required, such as the propagation time, to complete a given task.  
 We introduce a new technique to establish sharp upper bounds in terms of graph order  for  sum throttling and initial cost product throttling for power domination.
Furthermore,  we establish sharp bounds on possible changes of the product throttling number, both with and without initial cost, caused by certain graph operations for standard zero forcing, positive semidefinite forcing, and power domination.
\end{abstract}

\noi {\bf Keywords} Throttling, product throttling, zero forcing, power domination

\noi{\bf AMS subject classification} 05C57, 05C69, 68R10

\section{Introduction}

Throttling in graphs optimizes a sum or product of resources used and time required to complete a task.  In the case of several graph processes that can be viewed as coloring games on the vertices of the graph, we seek to minimize the combination of the number of vertices in the initial set used to start the process and the propagation time.  The study of (sum) throttling was initiated by Butler and Young in \cite{BY13}.  Two forms of product throttling, which optimizes the product of the resources and time, were introduced in \cite{cop-throttle2} for Cops and Robbers and in \cite{product1} for power domination. Anderson et al. studied both types of product throttling from a universal perspective and for the parameters standard zero forcing, positive semidefinite (PSD) forcing, power domination, and Cops and Robbers in \cite{product2}.  
In Theorem \ref{thm:6/7} we introduce a new technique to establish a sharp upper bound in terms of graph order for initial cost product throttling for power domination,  answering Question 2.66 in \cite{product2} in the affirmative.
In Theorem \ref{thm:limsuptake2} we apply this technique to establish a sharp upper bound in terms of graph order for the (sum) throttling number for power domination, thereby resolving the issue of the highest possible value of $\thpd(G)$ for a connected graph $G$ of order $n$ that was discussed in \cite[page 249]{HLSbook}.
In Section \ref{s:graoh-ops} we establish sharp bounds on possible changes caused by certain graph operations in both types (with and without initial cost) of product throttling number for standard zero forcing, PSD forcing, and power domination.
The remainder of this introduction contains some basic graph notation, precise definitions of sum and product throttling, the parameters studied,  and elementary results that will be used.

Given a (simple, undirected) graph $G$, let $V(G)$ denote the set of vertices and $E(G)$ denote the set of edges; $|V(G)|$ is the \emph{order} of $G$  and an edge between vertices $u$ and $v$ is denoted by $uv$ or $vu$. Given $v\in V(G)$, the (open) neighborhood of $v$ in $G$ is $N_G(v)=\{u: \ uv\in E(G)\}$ and the closed neighborhood of $v$ in $G$ is $N_G[v]=N(v)\cup \{v\}$.  The degree of a vertex $v$ is $\deg_G(v) = |N_G(v)|$.
A vertex $\ell$ of a graph $G$ is    a \emph{leaf} if $\deg_G(\ell)=1$.   A vertex $u$ of a graph $G$ is   \emph{universal} if $u$ is adjacent to every other vertex of $G$. 
Given a set $X\subseteq V(G)$,  the \emph{induced subgraph}  $G[X]$ has $V(G[X])=X$ and $E(G[X])=\{uv: uv\in E(G)\mbox{ and }u,v\in X\}$.  

A \emph{path} in a graph $G$ is a sequence of distinct vertices $v_1,v_2,\dots,v_r$ such that for each $i$ with $1 \leq i \leq r-1$ we have $v_iv_{i+1} \in E(G)$.   
  A graph $G$ is  \emph{connected} if for each pair of vertices $u,w \in V(G)$ there exists a path $v_1,v_2,\dots,v_r$ with $v_1=u$ and $v_r=w$; a graph is \emph{disconnected} if it is not connected.   The \emph{components} of $G$ are its maximal connected subgraphs.  Also, the \emph{path} $P_n$ is the graph with $V(P_n)=\{v_1,\dots,v_n\}$ and $E(P_n)=\{{ v_iv_{i+1}}: i=1,\dots,n-1\}$. 
The complete graph of order $n$, which has an edge between every pair of vertices,  is denoted by $K_n$.
 For a graph $H$ and positive integer $r$,  the graph $H\circ rK_1$ (known as the \emph{corona} of $H$ with $rK_1$) is constructed from a graph $H$ by adding $r$ leaves to each vertex of $H$. The graph $K_1\circ (n-1)K_1$ is the \emph{star} of order $n$; it is the complete bipartite graph on sets of sizes $1$ and $n-1$ and is more commonly denoted by $K_{1,n-1}$. A \emph{spider}  is a tree that has exactly one vertex of degree at least three, called its \emph{center}.\footnote{In the literature, a spider is also called a \emph{generalized star}, and sometimes is not required to have a vertex of degree three or more.} The  spider of order $a_1+ \dots+a_k+1$ with the center adjacent to $k\ge 3$ paths
of orders  $a_1\ge \dots\ge a_k$ is denoted by $S(a_1,\dots,a_k)$; these paths are called the \emph{legs}.

\subsection{Parameters}

In this section we provide precise definitions of the parameters discussed throughout.
Standard zero forcing, positive semidefinite zero forcing, and power domination can all be thought of as  coloring games on a graph, where the goal is to fill (color) all the vertices  (starting with  each vertex filled or unfilled); unfilled vertices are filled by applying a {color change rule}.\footnote{Various terms have been used including blue and white vertices, but filled/unfilled has recently become standard due to its suitability for all media.} 
Standard and positive semidefinite zero forcing originated in combinatorial matrix theory, providing upper bounds on the nullity of certain sets of symmetric matrices whose off-diagonal pattern of nonzero entries is described by the given graph; 
standard zero forcing was also introduced in control of quantum systems. Power domination originated from the problem of optimal placement of monitoring units in an electrical network. More information on the origins of these parameters and relevant references can be found in \cite{HLSbook}.

Standard zero forcing uses the \emph{standard color change rule}:
\bit
\item[] If $w$ is the unique unfilled neighbor of a filled vertex $v$, then  fill  $w$.  
\eit
Positive semidefinite (PSD) zero forcing uses the \emph{PSD color change rule}:
\bit
\item[] Let $B$ be the set of (currently) filled vertices and let $W_1,\dots, W_k$ be the sets of vertices of the  components of $G[V(G)\setminus B]$.  If $v\in B$, $w\in W_i$, and $w$ is the only unfilled neighbor of  $v$ in $G[W_i\cup B]$, then fill $w$. 
\eit
Note that it is possible that there is only one component of $G[V(G)\setminus B]$, and in that case the effect of the PSD color change rule is the same as that of the standard color change rule.   
 Forcing using the PSD color change rule is also called \emph{PSD forcing}.
 Repeated application of the standard or PSD color change rule until no more vertices can be filled is called the \emph{standard zero forcing propagation process} or \emph{PSD forcing propagation process}.

 \emph{Power domination} begins with a \emph{domination step}, in which every neighbor of a filled vertex is filled.  After that first step, the standard color change rule is applied. This is called the \emph{power domination propagation process}.  Power domination has natural connections to domination, so we provide notation for that also: A set $S\subseteq V(G)$ is a \emph{dominating set of $G$} if every vertex of $G$ is in $S$ or is a neighbor of a vertex in $S$.  The \emph{domination number} is 
 $\gamma(G)=\min\{|S|: S \mbox{ is a dominating set of } G \}$.  Note that a set consisting of a universal vertex is a dominating set of size  one. 
 
A set $B$ of vertices is called a standard zero forcing set, PSD forcing set, or power dominating set, respectively, if starting with the vertices in $B$ filled and the remaining vertices unfilled, the respective propagation process can fill all the vertices.  The \emph{standard zero forcing number} $\Z(G)$ is the minimum cardinality of a standard zero forcing set, and similarly for the \emph{PSD forcing number} $\Zp(G)$ and \emph{power domination number} $\pd(G)$.
We will sometimes use $Y$ to denote one or more of the parameters standard zero forcing,  PSD forcing, or power domination and the $Y$-number  $Y(G)$ of a graph $G$; a \emph{$Y$-set}, where $Y$ is one of these parameters, is an initial  set of  filled vertices that can fill all the other vertices under the the propagation process for $Y$ (called a $Y$-propagation process), and a \emph{minimum $Y$-set} is a $Y$-set of minimum cardinality.

\subsection{Propagation time and throttling}

Throughout this section $Y$ represents one the parameters standard zero forcing, PSD forcing, or power domination.
Starting with a set $B\subseteq V(G)$, 
we define two sequences of sets, the sets $B^{(i)}$ of vertices that are filled during time step $i$   and  the sets $B^{[i]}$ of vertices that are filled after time step $i$. Thus $B^{[0]}=B^{(0)}=B$ is the initial  set of filled vertices and $B^{[i+1]}=B^{[i]}\cup B^{(i+1)}$. 
Assume $B^{(i)}$ and $B^{[i]}$ have been constructed. Then 
\[B^{(i+1)}=\{w: 
\mbox{ $w$ can be  filled (given that all vertices in $B^{[i]}$  are filled)}\}.\] 
 The   \emph{$Y$-propagation time} of $B\subseteq V(G)$, denoted by $\ptx(G,B)$, is the least $t$ such that $B^{[t]}=V(G)$; if $B^{[t]}\ne V(G)$ for all $t$, then $\ptx(G,B)=\infty$. 
 Define  the \emph{$k$-propagation time of $G$} for $Y$ by
\[ \ptx(G,k)=\min_{|B|= k}\ptx(G,B).\]  The \emph{$Y$-propagation time of  $G$}  is 
$\ptx(G)=\ptx(G,Y(G)).$

We are now ready to list the various parameters related to throttling.  We define the (sum or product) throttling number of a set, the $k$-throttling number using exactly $k$ initially filled vertices, and the throttling number of the graph, minimizing over all acceptable choices of $k$.  Let $G$ be a graph of order $n$. 

\emph{Sum throttling\footnote{What we call the sum throttling number was introduced first and is often just called  the throttling number in the literature.} for $Y$}: For a set $B\subseteq V(G)$ and $1\le k\le n$, 
\[\thx(G,B)=|B|+\ptx(G,B), \ \thx(G,k)=\min_{|B|=k}\thx(G,B),\mbox{ 
and }\thx(G)=\min_{1\le k\le n}\thx(G,k).\]

There are two kinds of product throttling, with and without initial cost. 
 These two kinds can be thought of as dealing with applications in which there is a cost to getting resources into position before the propagation process begins and those that do not have such a cost. 

\emph{Initial cost product throttling for $Y$:}   For  a graph $G$ of order $n$, a set $B\subseteq V(G)$, and $1\le k\le n$, 
\[\thxx(G,B)=|B|(1+\ptx(G,B)), \ \thxx(G,k)=\min_{|B|=k}\thxx(G,B),\mbox{ 
and }\thxx(G)=\min_{1\le k\le n}\thxx(G,k).\]

\emph{No initial cost product throttling for $Y$:}   For   a graph $G$ of order $n$ that  has an edge, a set $B\subseteq V(G)$, and $1\le k\le n-1$, 
\[\thxa(G,B)=|B|\ptx(G,B), \ \thxa(G,k)=\min_{|B|=k}\thxa(G,B),\mbox{ 
and }\thxa(G)=\min_{1\le k\le n-1}\thxa(G,k).\]

The results in the next remark have appeared in \cite[Section 11.1]{HLSbook} and \cite{product2}.

\begin{rem}\label{bds-u}
  Let $G$ be a graph of of order $n$ that has an edge. Since $1\le Y(G)<n$  for the parameters $Y=\Z, \Zp, \pd$, we have  $2\le Y(G)+1\le \thxx(G), \thx(G)\le n$.   
   Similarly, $1\le \thxa(G)\le n-1$. 
\end{rem}


\section{New upper bounds for power domination throttling}\label{s:limsup}

  In this section we establish  new, sharp  upper bounds on the power domination  initial cost product and sum throttling numbers for all connected graphs. The product throttling bound answers  Question 2.66 in \cite{product2}. 
 We make use of the next bound on the domination number, which is well-known.  
  

\begin{thm}\label{dom-bd}{\rm  \cite{Ore}, \cite[Theorem 4.21]{core-graph-dom-book}}
If $G$ is a graph of order $n$ with no isolated vertices, then 
    $\gamma(G)\le \frac n 2$.
\end{thm}

 To establish the main theorems of this section, we need some  additional definitions. 
      A set $D$ is an \emph{edge-maximum minimum dominating set} of $G$ if $D$ is a minimum dominating set and  $G[D]$ has the maximum number of edges over all minimum dominating sets $D$ of $G$. 
  For a graph $G$,  a dominating set $D$ of $G$, and $v\in D$, a vertex $w$ is an \emph{external private neighbor} of $v$ relative to $D$ if  $w\not\in D$, $w\in N_G(v)$,  
and  $w\not\in N_G(z)$ for all $z\in D\setminus\{v\}.$  The set of external private neighbors of $v$ relative to $D$ is denoted by $\epn[v,D]$.
  Proposition 6 of \cite{BC79} states that if $G$ has no isolated vertices, there exists a minimum dominating set $D$ such that every vertex of $D$ has an external private neighbor (although the notation is different). The next lemma assumes connectedness and includes the hypothesis used to find such a $D$ in the proof of Proposition 6 of \cite{BC79}.  We include its  brief proof here for completeness and to highlight the subtle difference between this lemma and Lemma \ref{thm:maximin}.

\begin{lem}\label{thm:edge-maximum}
  Let $D$ be an edge-maximum minimum dominating set of a connected graph $G$ of order at least two. Then $\epn[v,D]\neq\emptyset$ for every vertex $v\in D$.
\end{lem}

\begin{proof}
    Suppose to the contrary
     that there exists a vertex $v\in D$ with $\epn[v,D] = \emptyset$.  If $v$ were adjacent to another vertex in $D$, then $D\setminus\{v\}$ would be a smaller dominating set of $G$, so $v$ is an isolated vertex in $G[D]$.  Since $G$ is connected, $v$ is adjacent to some vertex $u\in V(G)\setminus D$.  Now, $D^* = (D\setminus\{v\}) \cup \{u\}$ is a minimum dominating set of $G$, and $|E(G[D^*])| > |E(G[D])|$ since $u$ has at least one neighbor in $D$ other than $v$.
\end{proof}

Given $X\subseteq V(G)$, define $\Sigma(X)=\sum_{x\in X}\deg_G(x)$. 
 We define a dominating set $D$  to be \emph{optimal} if it is an  edge-maximum minimum dominating set that satisfies
\[
\Sigma(D) = \max\{\Sigma(D') : D' \text{ is an edge-maximum minimum dominating set}\}.
\]

\begin{lem}\label{thm:maximin}
Let $D$ be an optimal dominating set of a connected graph $G$   of order at least  three. For each $v\in D$, choose $u_v\in \epn[v,D]$ and define $A=\{u_v: \ v\in D\}$. Then $G[V(G)\setminus A]$ contains no isolated vertices. 
\end{lem}

\begin{proof}
    Suppose to the contrary that $G[V(G)\setminus A]$ has an isolated vertex $x$.
     Then $x \in D$ because otherwise it would not be isolated, since $D$ is a dominating set. 
    By the definition of $A$, $u_x$ is the unique neighbor of $x$ in $A$, so  $N_{G}(x)=\{u_x\}$. Since $G$ is connected and $|V(G)|\ge 3$, $u_x$ has degree at least two. Note that $D^*=(D\setminus\{x\})\cup \{u_x\}$ is a minimum dominating set of $G$.  Furthermore, $|E(G[D^*])| = |E(G[D])|$ since $x$ is isolated in $G[D]$. Since the degree of $u_x$ is strictly greater than the degree of $x$, we have $\Sigma(D^*)>\Sigma(D)$, contradicting the assumption that $D$ is an optimal dominating set. Thus, $G[V(G)\setminus A]$ has no isolated vertices.  
\end{proof}

Now we consider initial cost product throttling for power domination.  The next example slightly generalizes Example 2.65 in \cite{product2}.

\begin{ex}\label{ex:6n/7}
Let $H$ be a connected graph on vertices $v_1,v_2,\dots,v_{2k}$ let $G$ be given by
\[V(G)=V(H)\cup\{u_{i,1},u_{i,2}: \ 1\leq i\leq2k\}\cup\{u_{i,3}: \ 1\leq i\leq k\}\]
and 
\[E(G)=E(H)\cup\{\{v_i,u_{i,1}\},\{v_i,u_{i,2}\}: \ 1\leq i\leq2k\}\cup\{\{u_{i,2},u_{i,3}\}: \ 1\leq i\leq k\},\]
as shown in Figure~\ref{fig:67-ex}. 
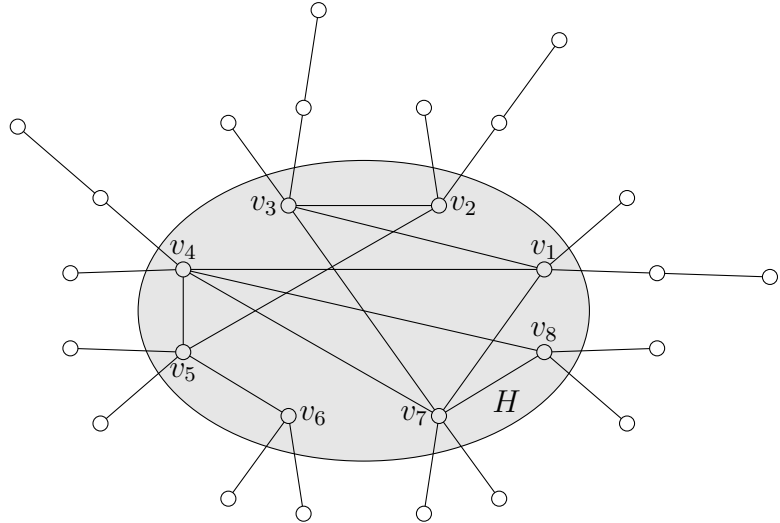
\begin{figure}[H]
\centering
\begin{tikzpicture}
\draw[fill=gray!20](0,0) ellipse (3cm and 2cm);\node at(1.9,-1.2){$H$};
\draw(1,-1.4)--(2.4,-0.55)--(-2.4,0.55)--(1,-1.4);\draw(-1,1.4)--(1,-1.4)--(2.4,0.55)--(-2.4,0.55)--(-2.4,-0.55);\draw(-1,-1.4)--(-2.4,-0.55)--(1,1.4)--(-1,1.4)--(2.4,0.55);
\draw(3.9,0.5)--(2.4,0.55)--(3.5,1.5);
\draw(1.8,2.5)--(1,1.4)--(0.8,2.7);
\draw(-1.8,2.5)--(-1,1.4)--(-0.8,2.7);
\draw(-3.9,0.5)--(-2.4,0.55)--(-3.5,1.5);
\draw(3.9,-0.5)--(2.4,-0.55)--(3.5,-1.5);
\draw(1.8,-2.5)--(1,-1.4)--(0.8,-2.7);
\draw(-1.8,-2.5)--(-1,-1.4)--(-0.8,-2.7);
\draw(-3.9,-0.5)--(-2.4,-0.55)--(-3.5,-1.5);
\draw(-3.5,1.5)--(-4.6,2.45);\draw(-0.8,2.7)--(-0.6,4);\draw(1.8,2.5)--(2.6,3.6);\draw(3.9,0.5)--(5.4,0.45);
\draw[fill=gray!20](2.4,0.55)circle(0.1);\node[above]at(2.4,0.55){$v_1$};
\draw[fill=white](3.9,0.5)circle(0.1);
\draw[fill=white](3.5,1.5)circle(0.1);

\draw[fill=gray!20](1,1.4)circle(0.1);\node[right]at(1,1.4){$v_2$};
\draw[fill=white](1.8,2.5)circle(0.1);
\draw[fill=white](0.8,2.7)circle(0.1);

\draw[fill=gray!20](-1,1.4)circle(0.1);\node[left]at(-1,1.4){$v_3$};
\draw[fill=white](-1.8,2.5)circle(0.1);
\draw[fill=white](-0.8,2.7)circle(0.1);

\draw[fill=gray!20](-2.4,0.55)circle(0.1);\node[above]at(-2.4,0.55){$v_4$};
\draw[fill=white](-3.9,0.5)circle(0.1);
\draw[fill=white](-3.5,1.5)circle(0.1);

\draw[fill=gray!20](2.4,-0.55)circle(0.1);\node[above]at(2.4,-0.55){$v_8$};
\draw[fill=white](3.9,-0.5)circle(0.1);
\draw[fill=white](3.5,-1.5)circle(0.1);

\draw[fill=gray!20](1,-1.4)circle(0.1);\node[left]at(1,-1.4){$v_7$};
\draw[fill=white](1.8,-2.5)circle(0.1);
\draw[fill=white](0.8,-2.7)circle(0.1);

\draw[fill=gray!20](-1,-1.4)circle(0.1);\node[right]at(-1,-1.4){$v_6$};
\draw[fill=white](-1.8,-2.5)circle(0.1);
\draw[fill=white](-0.8,-2.7)circle(0.1);

\draw[fill=gray!20](-2.4,-0.55)circle(0.1);\node[below]at(-2.4,-0.55){$v_5$};
\draw[fill=white](-3.9,-0.5)circle(0.1);
\draw[fill=white](-3.5,-1.5)circle(0.1);

\draw[fill=white](-4.6,2.45)circle(0.1);
\draw[fill=white](-0.6,4)circle(0.1);
\draw[fill=white](2.6,3.6)circle(0.1);
\draw[fill=white](5.4,0.45)circle(0.1);
\end{tikzpicture}
    \caption{A graph $G$ that achieves the 
    upper bound in Theorem \ref{thm:6/7}. 
  }  \label{fig:67-ex}
\end{figure}

 Any power dominating set $P$ of $G$ must contain at least one vertex from $\{v_i,u_{i,1},u_{i,2},u_{i,3}\}$ for each $1\leq i\leq k$ and one vertex from $\{v_i,u_{i,1},u_{i,2}\}$ for each $k+1\leq i\leq2k$, so $|P|\geq2k$. If $|P|<3k$, then there exists $1\leq i\leq k$ such that $P$ contains only one vertex from $\{v_i,u_{i,1},u_{i,2},u_{i,3}\}$, implying that $\ppt(G,P)\geq2$, so $\thpdx(G,P)=|P|(1+\ppt(G,P))\geq2k(1+2)=6k$. If $|P|\geq3k$, then $\thpdx(G,P)\geq3k(1+1)=6k$. Therefore, $\thpdx(G)\geq6k$. Noting that the order of $G$ is $7k$ and $\thpdx(G,V(H))=6k$, we have $\thpdx(G)=\frac{6}{7}|V(G)|$.
\end{ex}

It was asked in Question 2.66 in \cite{product2} if $\thpdx(G)\le\frac{6n}{7}$ for  connected graphs of arbitrarily large order $n$. The next result answers this question  in the  affirmative. 
Note that the result is not true for $n=1,2$ since a graph $G$ of order $n=1,2$ has $\thpdx(G)=n$, nor is it true without the assumption of connectivity, since $\thpdx(\ol{K_n})=n$.

\begin{thm}\label{thm:6/7}
Let $G$ be  a connected graph of order $n\ge 3$. Then $\thpdx(G)\le\frac{6n}{7}$  and this bound is sharp.
\end{thm}

\begin{proof}
Let $D$ be an optimal dominating set of $G$.  If $|D|\leq\frac{3n}{7}$,  then 
$\thpdx(G)\leq\thpdx(G,D)=|D|(1+\ppt(G,D))\leq\frac{6n}{7}$, since $\ppt(G,D)=1$.  So assume $|D|>\frac{3n}{7}$, and let $|D|=\frac{3n}{7}+\epsilon$ for some $0<\epsilon\leq\frac{n}{14}$, where the upper bound comes from Theorem \ref{dom-bd}.  
 In the steps that follow, we create a power dominating set $P$ such that $\thpdx(G,P)<\frac{6n}{7}$.

By Lemma~\ref{thm:edge-maximum}, for every $v\in D$  there exists a vertex $u_v\in\epn[v,D]$. Let $A=\{u_v:\ v\in D\}$. 
Then $|A|=\frac{3n}{7}+\epsilon$.
By Lemma~\ref{thm:maximin}, $G[ V(G) \setminus A]$ contains no isolated vertices. Let $P$ be a minimum dominating set of $G[ V(G) \setminus A]$, which contains no more than half of the vertices in $ V(G) \setminus A$  by Theorem \ref{dom-bd}. After one  time step  of power domination (i.e., the domination step) in $G$ using the set $P$, all vertices in $ V(G) \setminus A$ are filled, and each unfilled vertex in $A$ is the unique unfilled neighbor of some vertex in $D$. Hence, every vertex of $G$ is filled in at most two time steps  
of power domination. 

From our construction,   $|P| \leq \frac{|V(G) \setminus A|}{2} = \frac{n - (3n/7 + \epsilon)}{2} $ and $\ppt(G,P) \le 2$. 
Therefore, $\thpdx(G,P)=|P|(1+\ppt(G,P))<\frac{6n}{7}$.
Example \ref{ex:6n/7} shows that the bound is sharp.
\end{proof}

 An examination of the proof of Theorem \ref{thm:6/7} establishes the  next result. 
 
\begin{cor}\label{c:th67gamm37}
Let $G$ be a connected graph of order $n$. If $\thpdx(G)=\frac{6n}{7}$, then $\gamma(G)=\frac{3n}{7}$.
\end{cor}

The converse of Corollary \ref{c:th67gamm37} is false, as seen in the next example.  

\begin{ex}
  Let $G$ be the spider $S(2,2,1,1)$ and let $c$ denote the vertex of degree four.   Then $|V(G)| = 7$ and $\gamma(G) = 3$, since each leaf can be dominated only by itself or its neighbor.  
  However, $\thpdx(G) \leq \thpdx(G,\{c\}) = 1(1 + 2) = 3 < 6$. In fact, $\thpdx(G) = 3$.
\end{ex}

As discussed in Example 11.57 of \cite{HLSbook}, $\thpd(G)= \lf \frac n 3 \rf +2$ when  $G$ is a graph constructed    by adding a new leaf  to one leaf of $H\circ 2K_1$ where $H$ is  a connected graph of order at least two. 
We use a technique similar to that in the proof of Theorem \ref{thm:6/7} to show that $\lf \frac n 3 \rf +2$ is an upper bound on sum throttling for power domination (and is necessarily sharp), resolving the issue of the highest possible value of $\thpd(G)$ for a connected graph $G$ of  order $n$, which was discussed in \cite[page 249]{HLSbook}.

\begin{thm}\label{thm:limsuptake2}
Let $G$ be a connected graph of order $n$. Then  $\thpd(G)\le \lf \frac n 3 \rf +2$  
and this bound is sharp.
\end{thm}

\begin{proof}
 If $n=1$  or 2, the result is immediate, so assume $n\ge 3$. 
Let $D$ be an optimal dominating set of $G$. If $|D|\leq\frac{n}{3}$, then $\thpd(G)\leq\thpd(G,D)=|D|+1\leq\frac{n}{3}+1$, so $\thpd(G) \le \lf \frac n 3 \rf +1$ by taking the floor of both sides. Suppose $|D|>\frac{n}{3}$; then $|D|\geq \frac{n}{3}+\frac{1}{3}$. Since $|D|\leq\frac{n}{2}$, we let $|D|=\frac{n}{3}+\frac{1}{3}+\epsilon$ for some $0\leq\epsilon\leq\frac{n}{6}-\frac{1}{3}$. We create a power dominating set $P$ such that $\thpd(G,P)\le \lf \frac n 3 \rf +2$.

By Lemma~\ref{thm:edge-maximum}, for every vertex $v\in D$ there exists a vertex $u_v\in\epn[v,D]$. Let $A=\{u_v:\ v\in D\}$. 
Then $|A|=\frac{n}{3}+\frac{1}{3}+\epsilon$.
Let $P$ be a minimum dominating set of $G[ V(G)\setminus A]$. By the same argument given in the proof of Theorem \ref{thm:6/7},  the propagation time for $P$ is at most two.
From our construction, $|P|\le \frac{ |V(G) \setminus A|}{2} = \frac{n - (n/3 + 1/3 + \epsilon)}{2} = \frac{n}{3}-\frac{1}{6}-\frac{\epsilon}{2}$.  
Therefore, $\thpd(G)\leq\frac{n}{3}-\frac{1}{6}-\frac{\epsilon}{2}+2\le \frac{n}{3}+2$.  
 This implies  $\thpd(G)\le \lf\frac n 3\rf + 2$. 
\end{proof}



\section{ Graph operations for product throttling}\label{s:graoh-ops}

In this section we establish sharp bounds on the effect of edge deletion, vertex deletion, edge contraction, and edge subdivision for both types of product throttling for the parameters power domination, PSD forcing, and standard zero forcing.

Let $G$ be a graph,  $x$ be a vertex of $G$, and  $e=uv$ be an edge of $G$. The graph obtained from $G$ by deleting $x$ and all edges incident to $x$ is denoted by $G-x$, and the graph obtained from $G$ by deleting $e$ is denoted by $G-e$. The  \emph{contraction} $G/e$ of $G$ by $e$ is the graph obtained from $G$ by  identifying $u$ and $v$ (\emph{contracting} $e$) and suppressing any loops or multiple edges that arise in this process. 
The \emph{subdivision} of   $e$ in $G$, denoted by $G_e$, is the graph obtained from $G$ by adding a new vertex $z$, adding two edges $uz$ and $vz$, and removing the edge $e$.

 The next result might be described as folklore --- the easy proofs   have appeared in various places, although the results being proved were not always stated in this form. The proof of Proposition 11.9 in \cite{HLSbook} establishes \eqref{EC1}, \eqref{EC2},  and   \eqref{EC1-subdiv}; Conrad also proved \eqref{EC2}  for PSD propagation time in  \cite{C22}.  
Remark 11.4 in \cite{HLSbook} covers \eqref{EC1v}, and the proof of Proposition 11.8 in \cite{HLSbook} establishes  \eqref{EC1-contract} and \eqref{EC2-subdiv}. Theorem 8 in \cite{DVV16} basically proves \eqref{EC2-contract}, but since the terminology is very different, we provide a brief proof of 
\eqref{EC2-contract}.

\begin{lem}\label{l:EC}
    Let  $Y$ be one of the parameters $\pd, \Zp, \Z$. Let $G$ be a graph,  $x$ be a vertex of $G$,   $e=uv$ be an edge of $G$,   $y_e$ be the vertex formed by contracting $e$, and   $z_e$ be the vertex added by subdividing $e$. 
    
    \ben[$(1)$]
      \item\label{EC1} If $B'$ is a $Y$-set for $G-e$, then $\ptx(G, B'\cup\{w\})\le \ptx(G-e,B')$ for some $w\in\{u,v\}$.
      \item\label{EC2} If $B$ is a $Y$-set for $G$, then 
    $\ptx(G-e, B\cup\{w\})\le \ptx(G,B)$  for some $w\in\{u,v\}$.
      \item\label{EC1v} If $B'$ is a $Y$-set for $G-x$, then $\ptx(G, B'\cup\{x\})\le \ptx(G-x,B')$.
    \item\label{EC1-contract} 
 If $B'$ is a $Y$-set for $G/e$, then $\ptx(G, B)\le \ptx(G/e,B')$  where $B=B'\setminus \{y_e\}\cup \{u,v\}$ if ${y_e}\in B$ and $B= B'\cup\{w\}$ for some $w\in\{u,v\}$ otherwise. 
 \item\label{EC2-contract} 
If $B$ is a $\pd$-set for $G$, then $\ppt(G/e, B\setminus\{u,v\}\cup\{y_e\})\le \ppt(G,B)$  (there is no assumption that $u,v\in B$). 
\item\label{EC1-subdiv} 
If $B'$ is a $Y$-set for $G_e$, then 
$\ptx(G, B)\le \ptx(G_e,B')$
where  $B=B'\setminus\{z_e\}\cup\{w\}$ for some $w\in\{u,v\}$ if ${z_e}\in B$ and  $B=B'$   otherwise.
\item\label{EC2-subdiv}  If $B$ is a $Y$-set for $G$, then   $\ptx(G_e, B\cup\{w\})\le \ptx(G,B)$  for some $w\in\{u,v\}$. 
\een
\end{lem}
 \bpf \eqref{EC2-contract}: Let $B$ be a power dominating set of $G$. The proof of Theorem 8 in \cite{DVV16} shows that $B'=B\setminus\{u,v\}\cup\{y_e\}$ is a power dominating set of $G/e$.  Specifically every vertex that was dominated by $B$ in $G$ is now dominated by $B'$ in $G/e$.  Since $N_G[u]$ and $N_G[v]$ are dominated by $B'$ in $G/e$, the forcing process can proceed as in $G$ and the propagation time will not increase. 
\epf

Each of the inequalities in the next result is sharp as seen in the    subsequent examples.  

\begin{prop}\label{uprodth-edgedel}
    Let  $Y$ be one of the parameters $\pd$ or $\Zp$, let $G$ be a graph of order $n$, let $x$ be a vertex of $G$, and let $e=uv$ be an edge of $G$.

  \ben[$(1)$]
  \item\label{edge}   $\frac{1}{2} \thxa(G) \leq \thxa(G-e)\le 2\thxa(G)\quad \mbox{and}\quad\frac{1}{2} \thxx(G) \leq \thxx(G-e)\le 2\thxx(G).$
  \item\label{vertex}  $\frac{1}{2} \thxa(G) \leq \thxa(G-x)\quad \mbox{and}\quad\frac{1}{2} \thxx(G) \leq \thxx(G-x).$
   \item\label{uprodth-edgeops-contract-pd}  $\frac{1}{2} \thpda(G) \leq \thpda(G/e)\le  2\thpda(G) \quad \mbox{and}\quad\frac{1}{2} \thpdx(G) \leq \thpdx(G/e)\le  2 \thpdx(G).$  
   \item\label{uprodth-edgeops-contract-Zp}  $\frac{1}{2} \thpa(G) \leq \thpa(G/e)\quad \mbox{and}\quad\frac{1}{2} \thpx(G) \leq \thpx(G/e).$ \item\label{uprodth-edgeops-2} $\thxa(G) \leq \thxa(G_e)\le 2\thxa(G)$.
    \item\label{uprodth-edgeops-3} $\thpdx(G) \leq \thpdx(G_e)\le 2\thpdx(G)
   \quad \mbox{and}\quad  \thpx(G) \leq \thpx(G_e)\le \frac 3 2\thpx(G)$.
  \een
\end{prop}

 \bpf All but one of the inequalities  follow by choosing an initial $Y$-set to achieve the product throttling number, adjusting the $Y$-set (possibly increasing its size by one)  as specified in Lemma \ref{l:EC} to fill all vertices without increasing the propagation time, and noting that 
 the numerically most extreme case is when the initial set has only one vertex, so one additional vertex doubles its size. The exception is $ \thpx(G_e)\le \frac 3 2\thpx(G)$ in 
 \eqref{uprodth-edgeops-3}.
   
 Let $B$ be a $\Zp$-set for $G$ such that $\thpx(G) = \thpx(G, B)$.
It is shown    in \cite[Prop 10.37]{HLSbook} that $B$ is also a $\Zp$-set for $G_e$ with $\ptp(G_e, B) \leq \ptp(G, B) + 1$.  If $\ptp(G, B)\ge 1$, then
\[\small \thpx(G_e) \leq |B|(\ptp(G_e, B) + 1) \leq |B|(\ptp(G, B) + 2) = \frac{\ptp(G, B) + 2}{\ptp(G, B)  + 1}\thpx(G) \leq \frac{3}{2}\thpx(G). \]  
 If $\ptp(G, B)=0$, then $B=V(G)$ and setting $B'=V(G_e)$ gives  $\thpx(G_e,B')=n+1$; since $G$ has an edge, $ \thpx(G_e) \leq  \left( \frac {n+1}n\right) \thpx(G) \leq \frac{3}{2}\thpx(G)$. \epf

The bounds in the previous result are also true for $\Z$ but not useful since there are better bounds for each of the parameters except $\thzx$ (see Proposition \ref{z-k-induced} and Equation \eqref{eq:no-thzx}).

Let $G$ be a graph.  For a power dominating set or PSD forcing set $B\ne V(G)$  that achieves the relevant product throttling number, adding leaves to any vertex in $B$  does not change the product throttling number. Hence, each of the graphs in Examples  \ref{ex:edge-lower}--\ref{ex:upper-subdiv} can be expanded to an arbitrarily large order with the same product throttling numbers. 

 \begin{ex}\label{ex:edge-lower}         
        To see that  the lower bounds for $\thxa(G-e)$ and $\thxx(G-e)$  in Proposition \ref{uprodth-edgedel}\eqref{edge} are sharp with $Y\in\{\pd,\Zp\}$, 
        let $G$ be the graph shown in  Figure \ref{f:uprodth-edgeops-LB} including the dotted edge $e$; $G-e=S(3,3,3,3,3,3)$, the spider with six legs each of order three.  Then $\thxa(G-e)\le 3$ by choosing $B'=\{c\}$ as the  $Y$-set, and $\thxa(G)\le 6$ by choosing  $ B=\{c,u\}$ where $u$ is the endpoint of  $e$ not adjacent to $c$.  
        Similarly, $\thxx(G-e)\le 4$ and $\thxx(G)\le 8$ with the same $B'$ and $B$.  We show that no lower values can be obtained: The $Y$-sets $B$  and $B'$ have propagation time $3$. For $G-e=S(3,3,3,3,3,3)$, propagation time 2 requires  at least six vertices and propagation time  1  requires at least  seven  vertices for  $Y\in\{\pd,\Zp\}$.  For $G$, achieving propagation time less than $ 3$  requires seven  vertices for $Y=\Z_+$ and six vertices for $Y=\gamma_P$.       
  \begin{figure}[h!]
     \centering
       
   \begin{tikzpicture}
     \draw (180:3) -- (0:3);
     \draw (-60:3) -- (120:3);
     \draw (60:3) -- (-120:3);
     \draw[dotted, very thick] (0.5,0.866) .. controls (40:2.125) .. (1.5,2.598);
     \foreach \d in {0,60,...,300}{
        \foreach \r in {0,...,3}{
            \draw[fill=white] (\d:\r) circle (0.1);
        }
     }
     \draw (40:2.125) node{$e$};
     \draw (0,-0.1) node[anchor=north]{$c$};
     \draw (60:3) node[anchor=south west]{$u$};
     \end{tikzpicture}
     
      \caption{
     Adding the dotted edge $e$ to the spider $S(3,3,3,3,3,3)$ 
     doubles $\thxa$ and $\thxx$ for $Y\in \{\pd,\Zp\}$. } 
     \label{f:uprodth-edgeops-LB}
 \end{figure}
       \end{ex}

Recall from Remark \ref{bds-u} that the lowest possible values of $\thpda(G)$ and $\thpdx(G)$  are one and  two, respectively  (provided $G$ has at least one edge). Furthermore, 
   $\thpda(G)=1$ if and only if $G$ has a universal vertex if and only if  $\thpdx(G)=2$. \cite{product2}.
\begin{ex}  Upper bounds in Proposition~\ref{uprodth-edgedel}\eqref{edge} edge deletion: 
To show that the upper bound is sharp for $\thpda$, let $H_1$ be the graph shown in  Figure \ref{f:uprodth-edgeops-UB}(a) including the dotted edge $e$.  Then $\thpda(H_1)=1$ by choosing $B=\{c\}$ as the   power dominating set, and $\thpda(H_1-e)\ge 2$ because $H_1-e$ does not have a universal vertex (in fact, $\thpda(H_1-e)= 2$).     

For $\thpdx$,  let $H_2$ be the graph shown in  Figure \ref{f:uprodth-edgeops-UB}(b) including the dotted edge $e$. Then $\thpdx(H_2)=3$ by choosing $B=\{c\}$ as the power dominating set ($H_2$ does not have a universal vertex), and $\thpdx(H_2-e)= 6$ because $\pd(H_2-e)=2$, $\ppt(H_2-e,2)=2$, and $\ppt(H_2-e,3)=1$.   
\begin{figure}[hbt!]
\centering
\begin{subfigure}[b]{0.225\textwidth}
    \centering
    \begin{tikzpicture}
    \draw(1,0)--({cos(360/6)},{sin(360/6)});
    \foreach \i in {1,2,...,5}{\draw(0,0)--({cos(\i*360/6)},{sin(\i*360/6)});};
    \draw[dotted,very thick](0,0)--(1,0);
    \foreach \i in {0,1,...,5}{\draw[fill=white]({cos(\i*360/6)},{sin(\i*360/6)})circle(0.1);};
    \draw[fill=white](0,0)circle(0.1);
    \node[below]at(0.55,0.05){$e$};
    \node[above left]at(-0.1,-0.1){$c$};
\end{tikzpicture}
\caption{$H_1$}
\end{subfigure}
\begin{subfigure}[b]{0.225\textwidth}
\centering
\begin{tikzpicture}
    \draw(-1,0)--(2,0);
    \draw(-1,-1)--(1,-1);
    \draw(0,0)--(0,-1);
    \draw[dotted,very thick](0,0)--(1,-1);
    \foreach \x in {-1,0,1}{
    \draw[fill=white] (\x,0) circle (0.1);
    \draw[fill=white] (\x,-1) circle (0.1);
    }
    \foreach\i in {60, 120}{
        \draw (0,0) -- (\i:1);
        \draw[fill=white](\i:1)circle(0.1);
    }
    \draw[fill=white] (0,0) circle (0.1);
    \draw[fill=white] (2,0) circle (0.1);
    \node[below left]at(0,0){$c$};
    \node[above right]at(0.5,-0.65){$e$};
\end{tikzpicture}
\caption{$H_2$}
\end{subfigure}
\begin{subfigure}[b]{0.45\textwidth}
\centering
\begin{tikzpicture}
     \draw (180:3) -- (0:3);
     \draw (-60:3) -- (120:3);
     \draw (-120:3) -- (60:1);
     \draw[dotted, very thick] (60:1) -- (60:2);
     \draw (60:2) -- (60:3);
     \foreach \d in {0,60,...,300}{
        \foreach \r in {0,...,3}{
            \draw[fill=white] (\d:\r) circle (0.1);
        }
     }
  \node[below left]at(1.2,1.45){$e$};
     \draw (0,-0.1) node[anchor=north]{$c$};
  \node[below left]at(2.1,2.7){$u$};
\end{tikzpicture}
\caption{$H_3 = S(3,3,3,3,3,3)$}
\end{subfigure}
\caption{Each graph $H_i$ includes the dotted edge $e$. Deleting $e$ doubles the listed parameter(s): (a)  $\thpda$;  (b)  $\thpdx$;  (c)   $\thpa$, $\thpx$. 
  }\label{f:uprodth-edgeops-UB} 
\end{figure}


The graph $H_3=S(3,3,3,3,3,3)$ shown in  Figure \ref{f:uprodth-edgeops-UB}(c) (including the dotted edge $e$) shows that  the upper bounds are  sharp for $\thpa$ and  $\thpx$: As shown  in Example \ref{ex:edge-lower}, $\thpa(H_3)= 3$  and $\thpx(H_{3})= 4$. 
For $H_3-e$, choosing  $B'=\{c,u\}$ where $u$ is  
the leaf vertex on the leg of $H_3$ that contains $e$ gives $\thpa(H_{3}-e)\le 6$ and $ \thpx(H_{3}-e)\le 8$.    
For $H_3-e$, propagation time 2 requires  at least six vertices and propagation time  1  requires at least  seven  vertices, 
so $\thpa(H_{3}-e)= 6$ and $ \thpx(H_{3}-e)= 8$. 
 \end{ex}

The example $S(3,3,3,3,3,3)$ also demonstrates sharpness for power domination (for both types of product throttling). However, when available, examples with all graphs connected are preferred for throttling, so we presented $H_1$ and $H_2$.  There is no example with $H-e$ connected that achieves  $\thpa(H-e)=2\thpa(H)$ or  $\thpx(H-e)=2\thpx(H)$, because as seen in Lemma \ref{l:EC}\eqref{EC2}, $\ptp(H-e, k+1)\le \ptp(H,k)$.  Doubling  
the product throttling number requires $k=1$, which means $\Zp(H)=1$. 
This implies $H$ is a tree \cite[Theorem 9.45]{HLSbook}, so deleting an edge disconnects $H$. 

\begin{figure}[hbt!]
    \centering
    \begin{tikzpicture}
    \coordinate (c) at (225:0.3);
    \coordinate (x) at (45:0.3);
    \foreach \t in {0,90,180,270} {
    \coordinate (a) at (\t:1);
    \coordinate (b) at (\t:2);
        \draw (a) -- (b);
        \draw (c) -- (a);
        \draw (x) -- (a);
        \draw[fill=white] (a) circle (0.1);
        \draw[fill=white] (b) circle (0.1);
    }
    \draw[fill=white] (c) circle (0.1) node[anchor=north east]{$c$};
    \draw[fill=white] (x) circle (0.1) node[anchor=south west]{$x$};
    \end{tikzpicture}
    \caption{ Graph $G$ is constructed from spider $S(2,2,2,2)$ by adding an independent twin $x$ of the center vertex $c$, which doubles $\thxa$ and $\thxx$ for $Y\in\{\pd,\Zp\}$.
     } 
    \label{f:thpdx-v}
\end{figure}
\begin{ex} For sharpness of the vertex deletion bound in  Proposition \ref{uprodth-edgedel}\eqref{vertex}, let $G-x=S(2,2,2,2)$ and let $x$ be an independent twin of the center vertex $c$ of $G-x$; see Figure \ref{f:thpdx-v}.  For $Y\in\{\pd,\Zp\}$, $\thxa(G)=4$ and  $\thxx(G)=6$ (using $B=\{c,x\}$), whereas $\thxa(G-x)=2$ and $\thxx(G-x)=3$   (using $B'=\{c\}$); in each case a minimum $Y$-set is used
 and the propagation time is 2.   Propagation time  1 
requires at least  four  vertices for $G$ and $G-x$.   

 \end{ex}


\begin{ex} For sharpness of the lower bounds in Proposition \ref{uprodth-edgedel}\eqref{uprodth-edgeops-contract-pd} regarding edge contraction for power domination and the first lower bound in Proposition \ref{uprodth-edgedel}\eqref{uprodth-edgeops-contract-Zp} regarding edge contraction for PSD forcing, 
let $G=K_2\circ 2K_1$ where $V(K_2)=\{c_1,c_2\}$ and  $e=c_1c_2$. Then $G/e\cong K_{1,4}$ with center   $y_e$, the contracted vertex (see Figure \ref{f:pd-econ}).     Thus $\thpda(G)=\thpa(G)=2$ and $\thpdx(G)=4$ (using $B=\{c_1,c_2\}$), while $\thpda(G/e)=\thpa(G/e)=1$ and $\thpdx(G/e)=2$ (using $B'=\{ y_e\}$). 
For power domination, the number of vertices  used is the power domination number and the propagation time is 1;  $\ptp(G,1)=2$. 
For sharpness of the second lower bound in \eqref{uprodth-edgeops-contract-Zp},  $ \frac{1}{2} \thpx(G) \leq \thpx(G/e)$, see Example \ref{ex:PSDinitcostedgecontract}.
 \end{ex}

\begin{figure}[hbt!]
     \centering
     \begin{subfigure}[b]{0.45\textwidth}
     \centering
     \begin{tikzpicture}
         \draw (-0.707,0.707) -- (0,0) -- (1,0) -- (1.707,0.707);
         \draw (-0.707,-0.707) -- (0,0);
         \draw (1,0) -- (1.707,-0.707);
         \draw[fill=white] (-0.707,0.707) circle (0.1);
         \draw[fill=white] (-0.707,-0.707) circle (0.1);
         \draw[fill=white] (0,0) circle (0.1);
         \draw[fill=white] (1,0) circle (0.1);
         \draw[fill=white] (1.707,0.707) circle (0.1);
         \draw[fill=white] (1.707,-0.707) circle (0.1);
         \draw (0.5,0) node[anchor=south]{$e$};
     \end{tikzpicture}
     \caption{}
     \label{f:d-star}
     \end{subfigure}
     \begin{subfigure}[b]{0.45\textwidth}
     \centering
     \begin{tikzpicture}
         \foreach \t in {45,135,225,315}{
            \draw (0,0) -- (\t:1);
            \draw[fill=white] (\t:1) circle (0.1);
         }
         \draw[fill=white] (0,0) circle (0.1);
         \draw (0,0.1) node[anchor=south]{$y_e$};
     \end{tikzpicture}
     \caption{}
     \label{f:d-star-ctr}
     \end{subfigure}
     \caption {(a) shows graph $G$ with edge $e$, and (b) shows $G/e$. Contracting $e$ halves $\thpda$, $\thpa$, and $\thpdx$. 
    }  
     \label{f:pd-econ}
 \end{figure}

\begin{ex} For sharpness of the upper bounds in Proposition \ref{uprodth-edgedel}\eqref{uprodth-edgeops-contract-pd} regarding edge contraction for power domination,   let $G$ be the graph shown in Figure \ref{f:thpdx-edgecon}(a); $G/e$ is shown in \ref{f:thpdx-edgecon}(b). Then $\thpda(G)=2$ and $\thpdx(G)=3$ (using $B=\{c\}$), while $\thpda(G/e)=4$ and $\thpdx(G/e)=6$ (using $B'=\{c,{y_e}\}$); in each case, the number of vertices  used is the power domination number and the propagation time is 2.   Propagation time  1  requires at least  five  vertices for $G$ and four for $G/e$.  
 \end{ex}

 \begin{figure}[hbt!]
     \centering
     \begin{subfigure}[b]{0.45\textwidth}
     \centering
     \begin{tikzpicture}
         \draw (0:1) -- (72:1);
         \draw (36:1.05) node{$e$};
         \foreach \t in {0,...,4}{
            \coordinate (a) at (360*\t/5:1);
            \coordinate (b) at (360*\t/5:2);
            \draw (0,0) -- (a) -- (b);
            \draw[fill=white] (a) circle (0.1);
            \draw[fill=white] (b) circle (0.1);
         }
         \draw[fill=white] (0,0) circle (0.1);
         \draw (180:0.35) node{$c$};
     \end{tikzpicture}
     \caption{}
     \label{f:thpdx-edgecon-a}
     \end{subfigure}
     \begin{subfigure}[b]{0.45\textwidth}
     \centering
     \begin{tikzpicture}
         \draw (0,0) -- (36:1);
         \draw (36:1) -- (54:2);
         \draw (36:1) -- (18:2);
         \draw[fill=white] (36:1) circle (0.1) node[anchor=north west]{$y_e$};
         \draw[fill=white] (54:2) circle (0.1);
         \draw[fill=white] (18:2) circle (0.1);
         \foreach \t in {2,3,4}{
            \coordinate (a) at (360*\t/5:1);
            \coordinate (b) at (360*\t/5:2);
            \draw (0,0) -- (a) -- (b);
            \draw[fill=white] (a) circle (0.1);
            \draw[fill=white] (b) circle (0.1);
         }
         \draw[fill=white] (0,0) circle (0.1);
         \draw (180:0.35) node{$c$};
     \end{tikzpicture}
     \caption{}
     \label{f:thpdx-edgecon-b}
     \end{subfigure}
     \caption{(a) shows graph $G$ with edge $e$, and (b) shows $G/e$. Contracting $e$ doubles $\thpda$ and $\thpdx$. 
    } 
     \label{f:thpdx-edgecon}
 \end{figure}

 \begin{ex}\label{ex:PSDinitcostedgecontract}
     We can modify the graph $G$ in Figure \ref{f:thpdx-edgecon}(a)  to obtain graph $H$ by adding a leaf to each of the three legs of order two. Then $H$  shows that the second lower bound in Proposition \ref{uprodth-edgedel}\eqref{uprodth-edgeops-contract-Zp} is sharp:  
     $\thpx(H)=8$ and $\thpx(H/e)=4$. In each case a minimum PSD forcing set is used and the propagation time is 3.  To achieve propagation time less than 3, five vertices are needed for $H$ and four vertices are needed for $H/e$. 
 \end{ex}

\begin{ex}
Paths can be used to establish sharpness for  all lower bounds in  Proposition \ref{uprodth-edgedel} \eqref{uprodth-edgeops-2} and \eqref{uprodth-edgeops-3} regarding  edge subdivision.  For $Y\in\{\pd,\Zp\}$, it is known  (see \cite{product2}) that $\thxa(P_n) = \lc \frac{n}{3} \rc$ and $\thxx(P_n) = 1 + \lc \frac{n-1}{2} \rc$.    Note that 
 $(P_n)_e = P_{n+1}$ for any edge $e$.  If $n \not\equiv 0 \pmod 3$, then
    $\thxa(P_n) = \lc \frac{n}{3} \rc = \lc \frac{n+1}{3} \rc = \thxa(P_{n+1}).$   If $n$ is even, then
    $\thxx(P_n) = 1 +  \lc \frac{n-1}{2} \rc = 1 +  \lc \frac{n}{2} \rc = \thxx(P_{n+1}).$
\end{ex}

\begin{ex}\label{ex:upper-subdiv}
The star $K_{1,n-1}$ with $n\ge 3$ establishes the sharpness of the upper bounds for no initial cost product throttling for $\pd$ and $\Zp$ in Proposition \ref{uprodth-edgedel}\eqref{uprodth-edgeops-2}  regarding 
edge subdivision:
 For $n\ge 3$,    $\thxa(K_{1,n-1}) = 1$ by using the center  vertex as the initial set. If any edge is subdivided, with an initial set of size $1$ the propagation time will be $2$, and the only way propagation time could remain $1$ is if the initial set gets an extra vertex. In either case, $\thxa$ is doubled from $1$ to $2$.

The star also establishes the sharpness of the upper bound for initial cost product throttling for  $\Zp$ {in Proposition \ref{uprodth-edgedel}\eqref{uprodth-edgeops-3}}: For $n\ge 3$,  $\thpx(K_{1,n-1}) = 2$ with the initial set $B$ being the center vertex and propagation time 1. With the same initial set, the propagation time after an edge is subdivided increases to 2,
    and this cannot be improved because an initial set with more vertices at least doubles the throttling number. Therefore,
    $\thpx((K_{1,n-1})_e) = 1(1 + \ptp((K_{1,n-1})_e, B)) = 1(1 + 2) = 3.$


Finally, to show that $\thpdx(G_e) = 2\thpdx(G)$ is possible, 
consider  the graph $G$ in Figure \ref{fig:sharpness-subdiv}. 
    The  set $\{c\}$  is 
    a  power dominating set with propagation time $2$, giving $\thpdx(G) \leq 1(1 + 2) = 3$.
 This cannot be improved because any initial set that achieves propagation time $1$ must have at least
    $3$ vertices, so $\thpdx(G) = 3$. If the edge $e $ shown in Figure \ref{fig:sharpness-subdiv} is subdivided,
    $\{c\}$ is no longer a power dominating set. Since $c$ must be in any initial power dominating set $B$ that does not include at least three leaves adjacent to $c$, we have $|B| \geq 2$. Denoting the new vertex by $ z_e$, we can see that $\{c,{z}_e\}$ is a power dominating set, and  $\thpdx(G_e) \le \thpdx(G_e, \{c,{z}_e\}) = 2(1 + 2) = 6$.
    In order to reduce the propagation time to $1$,  the initial set must contain 3 vertices,
    which still yields a throttling number of $3(1 + 1) = 6$. Therefore, $\thpdx(G_e) = 6$. 
\begin{figure}[H]
\centering
\begin{tikzpicture}
\draw({1+sqrt(3)/2},-0.5)--({sqrt(3)/2},-0.5);
\draw({1+sqrt(3)/2},0.5)--({sqrt(3)/2},0.5);
\draw({sqrt(3)/2},0.5)--(0,0)--({sqrt(3)/2},-0.5)--cycle;
\foreach\i in{1,2,3,4}{\draw(0,0)--({cos(90+\i*36)},{sin(90+\i*36)});}
\node[right]at({sqrt(3)/2},0){$e$};
\node[below]at(0,0){$c$};
\draw[fill=white]({1+sqrt(3)/2},-0.5)circle(0.1);
\draw[fill=white]({1+sqrt(3)/2},0.5)circle(0.1);
\draw[fill=white]({sqrt(3)/2},-0.5)circle(0.1);
\draw[fill=white]({sqrt(3)/2},0.5)circle(0.1);
\draw[fill=white](0,0)circle(0.1);
\foreach\i in{1,2,3,4}{\draw[fill=white]({cos(90+\i*36)},{sin(90+\i*36)})circle(0.1);}
\end{tikzpicture}
\caption{A graph $G$ and edge $e$ such that subdividing edge $e$ doubles  $\thpdx$. 
}
\label{fig:sharpness-subdiv}
\end{figure}
 \end{ex}

 There are good reasons that only one bound is presented for vertex deletion in Proposition \ref{uprodth-edgedel}\eqref{vertex}.  Deleting a universal vertex can change $\pd$ and $\Zp$ and thus the associated product throttling numbers substantially, as illustrated in the next example. 

\begin{ex} 
    Let $H$ be a connected graph of order $r$. Let $G=((H\circ K_1)\circ K_1)\vee K_1$ be the graph of order $n=4r+1$ obtained by first adding a leaf to every vertex of $H$ to get $H\circ K_1$, then adding a leaf to every vertex of $H\circ K_1$ to get $(H\circ K_1)\circ K_1$, and then adding a universal vertex $u$.  Then $\thpda(G)=1$ and  $\thpdx(G)=2$, whereas $\thpda(G-u)=\frac{n-1}2$ \cite{product1} and  $\thpdx(G-u)=\frac{3(n-1)}4$ \cite{product2}.  
   For PSD forcing (allowing a disconnected graph), consider the graph $H'$ created from a star with $n-1$ leaves by adding an edge between two leaves; 
   $c$ is the center vertex. Then $\thpa(H')=2$ and   $\thpx(H')=4$,   whereas $\thpa(H'-c)=n-2$ and $\thpx(H'-c)= n-1$. 
\end{ex}

 A similar situation can occur with edge contraction for PSD forcing: It is shown in \cite[Example 9.53]{HLSbook} that $\Zp(B_k)=2$ and $\Zp(B_k/e)=k+1$ where $B_k=K_{1,k}\cp P_2$ is the book graph of order $2k+2$  and $e$ is the spine of the book (that is, $e = uv$ where $\deg(u) = \deg(v) = k+1$).  Since $\ptp(B_k,2)=1$, $\thpa(B_k)=2$ and $\thpx(B_k)=4$.  On the other hand, $\thpa(B_k/e)= k+1$ (since $\ptp(B_k/e,k+1)=1$) and $\thpx(B_k/e)= 2k+1$ (since $\Zp(B_k/e)\ge \frac{|V(B_k/e)|}2$).\vspace{5pt}
 
Finally  we consider  product throttling for standard zero forcing, which behaves differently than other types of  throttling.  It is shown in \cite{product2} that  \begin{equation}\label{eq:no-thzx}
    \thzx(G) = |V(G)|,\end{equation} which
is achieved by filling all vertices. 
Define $k(G,p)=\min\{k:\ptz(G,k)=p\}$; $k(G,1)$ is of particular importance due to the next result.  
\begin{thm}\label{nicp-z}{\rm\cite{product2}}
For any graph $G$, $\thza(G)$ is the least $k$ such that
$\ptz(G, k) = 1$, i.e. $\thza(G) = k(G, 1)$. Necessarily,
$k(G,1) \geq \frac{n}{2}$.
\end{thm}

Graphs $G$ attaining $\thza(G)= \frac n 2$ have been characterized and provide useful examples.
Let $G_1$ and $G_2$ be disjoint graphs of equal order and let $M$ be a matching between $V(G_1)$ and $V(G_2)$ that saturates all vertices. Then the  \emph{$M$-sum of $G_1$ and $G_2$}, denoted by $G_1M^+G_2$, is the graph with $V(G_1M^+G_2)=V(G_1)\cup V(G_2)$ and $E(G_1M^+G_2)=E(G_1)\cup E(G_2)\cup M$. A graph of the form $G_1M^+G_2$ is also called a \emph{matched-sum graph}.  Observe that the maximum number of edges in a matched-sum graph of order $n=2r$ 
is $r^2$, which is attained by $K_rM^+ K_r$. 

\begin{thm}\label{thza-halfeven}  {\rm \cite{product2}}
Let $G$ be a connected graph  of even order $n$.  Then $G$ satisfies $\thza(G)= \frac n 2$ if and only if $G$ is a connected matched-sum graph.
\end{thm}

As noted in \cite{product2}, the path $P_{2r}$ is a matched-sum graph, $\thza(P_{2r})=r$, and $\thza(P_{2r+1})=r+1$; recall that we denote the vertices of $P_n$ by $v_1,\dots,v_n$ in order.  When one vertex $u$ forces another vertex $v$, we sometimes denote this by $u\to v$.

 \begin{prop}\label{z-k-induced}   Let $G$ be a graph, let $x$ be a vertex of $G$,  let $e=uv$ be an edge. 
 \begin{enumerate}[$(1)$]
 \item Assuming  $G-{x}$ contains an edge, 
 $\thza(G)-1\le  \thza(G-{x})\le \thza(G)$.
    \item Assuming  $G-e$ contains an edge,   $\thza(G)-1\leq \thza(G-e) \leq \thza(G)+1$.
     \item Assuming  $G/e$ contains an edge, $\thza(G)-1\le  \thza(G/e)\le \thza(G)$.
     \item $\thza(G) \leq \thza(G_e) \leq \thza(G) + 1$. 
\end{enumerate}
These bounds are sharp.
  \end{prop}
  \bpf We establish the corresponding bounds for $k(G,1)$ and then apply Theorem \ref{nicp-z}.   We use matched-sum graphs in examples showing sharpness.  Let $y_e$ denote the vertex created by contracting edge $e$ and $z_e$ denote the vertex created by subdividing edge $e$.
  
  $k(G,1)\le  k(G-{x},1)+1$: Choose $B'$ such that $\ptz(G-{x},B')=1$ and $|B'|=k(G-x,1)$. Let $B=B'\cup \{x\}$.  Then the same propagation process can force in one round.   
  For  sharpness, consider $P_{2r+1}$ and let $x=v_{2r+1}$: $P_{2r+1}-{x}\cong P_{2r}$ and $\thza(P_{2r+1}-{x})=r=\thza(P_{2r+1})-1$.    

    $k(G-{x},1)\le k(G,1)$: Choose $B$ such that $\ptz(G,B)=1$ and $|B|=k(G,1)$.  If ${x}\in B$ and $x\to w$ when using $B$ to force all vertices in one round, then let $B'=B\setminus\{{x}\}\cup\{{w}\}$. If ${x}\in B$ and ${x}$ does not force when using $B$ to force all vertices in one round, then let $B'=B\setminus\{{x}\}$.  Otherwise let $B'=B$ and let ${w}$ be the vertex that forces ${x}$.   In all cases, the same propagation process can force in one round (deleting  $x\to w$  or  $w\to x$ if needed), and $ |B'|\le |B|$.  
    For sharpness, consider $P_{2r}$ and let $x=v_{2r}$:  $P_{2r}-{x}\cong P_{2r-1}$ and $\thza(P_{2r}-{x})=r=\thza(P_{2r})$.

  $k(G,1)\le  k(G-e,1)+1$: Choose $B'$ such that $\ptz(G-e,B')=1$ and $|B'|=k(G-e,1)$. 
  If exactly one of $u,{v}$ performs a force  when using $B'$ to force all vertices in one round, then assume (without loss of generality) that $u$ performs the force and  let $B=B'\cup\{{v}\}$.  When both force or neither force, let $B=B'$. 
   Then the same propagation process can force in one round and $ |B|\le |B'|+1$.   
   For sharpness, consider the matched-sum graph $G-e=K_rM^+ K_r$, which has order $2r$.  Then $\thza(G-e)=r$ and $G-e$ has $r^2$ edges.  Adding any edge (necessarily between the two copies of $K_r$) results in a graph with $r^2+1$ edges, so $\thza(G)\ge r+1$.

$k(G-e,1)\le  k(G,1)+1$: Choose $B$ such that $\ptz(G,B)=1$ and $|B|=k(G,1)$.   If one of $u,{v}$  forces the other  when using $B$ to force all vertices in one round, then assume (without loss of generality) that $u\to {v}$  and  let $B'=B\cup\{{v}\}$. 
Otherwise let $B'=B$.  Then the same propagation process can force in one round (deleting $u\to {v}$ if needed). 
    Thus $k(G,1)\le  k(G-e,1)+1$.  
 For sharpness, consider  $K_rM^+ K_r$ with $r\ge 3$ and note that $\thza(K_rM^+ K_r)=r$. Let $e$ be an edge in the matching between the two copies of $K_r$.  Then  $\thza(K_rM^+ K_r-e)=r+1$ because no set of $r$ vertices can force in one round. 

 
 $k(G,1)\le k(G/e,1)+1$: Choose $B'$ such that $\ptz(G/e,B')=1$ and $|B'|=k(G/e,1)$.  If ${y_e}\in B'$, then $B=B'\setminus \{{y_e}\}\cup\{u,v\}$.  If ${y_e}\not\in B'$, then there is a vertex $x$ such that $x\to {y_e}$; without loss of generality,  $x$ is adjacent to $u$ in $G$, and let {$B=B'\cup\{v\}$}.  
For sharpness, consider $K_{1,n-1}, n\ge { 3}$ 
and $e$ any edge, so $(K_{1,n-1})/e\cong K_{1,n-2}$.  Thus $\thza((K_{1,n-1})/e)=n-2=\thza(K_{1,n-1})-1$.
 
  $k(G/e,1)\le  k(G,1)$:  Choose $B$ such that $\ptz(G,B)=1$ and $|B|=k(G,1)$. If $u,{v}\in B$ and at least one of $u,v$ forces a vertex, say $u\to w$, then $B'=B\setminus\{u,{v}\}\cup\{{y_e,w}\}$ {(if $u\to w$ and $v\to w'$ in $G$, then $   y_e\to w'$ in the first round works for $G/e$)}.
  If $u,{v}\in B$ and neither $u$ or ${v}$ forces, then $B'=B\setminus\{u,{v}\}\cup\{{y_e}\}$. 
  If $u\in B$ and ${v}\not\in B$, then $u$ does not force or $u\to {v}$, so $B'=B\setminus\{u\}\cup\{{y_e}\}$.  If $u,{v}\not\in B$, then there are vertices $a,b$ such that $a\to u$ and $b \to v$, so let $B'=B$ and $a\to  y_e$ in the first round works.
For sharpness, consider $P_{2r}$ and $e=v_{2r-1}v_{2r}$: $(P_{2r})/e\cong P_{2r-1}$, so $\thza((P_{2r})/e)=r=\thza(P_{2r})$.

 $k(G_e,1)\le  k(G,1)+1$ and $k(G,1)\le k(G_e,1)$ follow from the contraction bounds since setting $H=G_e$ with new edges $f=u{z_e}$ and $vz_e$ and contracting $f$ gives $G\cong H/f$.  Paths can be used for sharpness examples.
  \epf

\section*{Acknowledgements}
This research began at the American Institute of Mathematics with support from NSF DMS grant  2015462.  The authors thank AIM and NSF. Ryan Blair was supported in part by NSF grant DMS-2424734. Veronika Furst was supported in part by NSF grant DMS-2331072.


\end{document}